\documentclass{amsart}
\usepackage{amsmath,amsfonts,amsthm}
\usepackage{amssymb,latexsym}
\usepackage{graphics}
\usepackage[colorlinks]{hyperref}

\usepackage[all]{xypic}

\theoremstyle{plain}
\theoremstyle{definition}
\newtheorem{theorem}{Theorem}[section]
\newtheorem{lemma}[theorem]{Lemma}

\newtheorem{corollary}[theorem]{Corollary}

\newtheorem{counterexample}[theorem]{Counterexample}

\newtheorem{remark}[theorem]{Remark}
\theoremstyle{remark}

\numberwithin{equation}{section}
\newcommand{\SP}{\: \: \: \: \:}

\title[On Fort transformation groups]{Indicator sequences and indicator topologies of Fort transformation groups}
\author[F. Ayatollah Zadeh Shirazi, F. Ebrahimifar]{Fatemah Ayatollah Zadeh Shirazi, Fatemeh Ebrahimifar}
\begin{document}
\begin{abstract}
In the following text we prove that there exists a Fort transformation group
with indicator sequence $(p_0,\ldots,p_n)$ if and only if $0=p_0\leq p_1\leq\cdots\leq p_n=1$,
moreover we characterize all possible indicator topological spaces of Fort transformation groups.
The text will study indicator sequences and indicator topologies of Fort transformation semigroups too.
\end{abstract}
\maketitle
\noindent {\small {\bf 2010 Mathematics Subject Classification:}  54H20 \\
{\bf Keywords:}} Alexandroff compactification, Fort space, Height, Indicator sequence, Indicator topology, Transformation group, Transformation semigroup.
\section{Introduction}
\noindent
``Transformation groups'' as one of the branches of topological dynamics is
one of the main interests of many specialists. One may be faced several ideas during
his/her studies in transformation groups, including classifying problems and particular spaces or conditions.
One may use several tools to classify the category of transformation groups, in this text we use
indicator sequences and indicator topologies \cite{indicator}. Also one may consider special
phase spaces like ``metric spaces'', ``unit interval'', ``compactifications'', ``Fort spaces'', etc. \cite{nemidoonam1, nemidoonam2, nemidoonam3}.
In this text we pay attention to transformation groups with a Fort space as phase space, let's recall that
as a matter of fact a Fort space is the smallest compactification of a discrete space, i.e. one point compactification
(or Alexandroff compactification) of a discrete space.
\section{Preliminaries}
\noindent By a transformation semigroup (group) $(Y,S,\pi)$ or simply $(Y,S)$, we mean a compact Hausdorff
space $Y$, a discrete topological semigroup (group) $S$ with identity $e$ and continuous map
$\pi:\mathop{Y\times S\to Y}\limits_{(x,s)\mapsto xs}$ such that:
\begin{itemize}
\item $\forall x\in Y\SP(xe=x)$,
\item $\forall x\in Y\:\:\forall s,t\in S\SP(x(st)=(xs)t)$.
\end{itemize}
In transformation semigroup $(Y,S)$ we say the nonempty subset $Z$ of $Y$
is invariant, if $ZS:=\{zs:z\in Z,s\in S\}\subseteq Z$. If $W$ is a closed invariant (nonempty)
subset of transformation semigroup $(Y,S)$, suppose $h(W)$ is the supremmum of
all numbers $n\geq0$ such that there exist distinct $Z_0\subset\cdots\subset Z_n=W$ of
closed invariant (nonempty) subsets of $(Y,S)$, we call $h(W)$ as {\it height of }$W$
(see \cite{height}).
We denote $h(Y)$ also by $h(Y,S)$ to make emphasis on
phase semigroup $S$.
\begin{remark}\label{rem10}
In transformation semigroup $(Y,S)$ with for $p\in\{0,1,2,\ldots\}$, the
following statements are equivalent \cite[Theorem 3.2]{indicator}:
\begin{itemize}
\item $h(Y,S)=p$,
\item there exists a maximal chain $Y_0\subset \cdots \subset Y_p=Y$ of distinct closed
invariant subsets of $(Y,S)$ of length $p+1$,
\item $\{\overline{yS}:y\in Y\}$ has exactly $p+1$ elements.
\end{itemize}
\end{remark}
\noindent Now in transformation semigroup $(Y,S)$ with finite height $h(Y,S)=p<+\infty$, by
Remark~\ref{rem10} there exists $y_0,\ldots,y_p\in Y$ such that
$\{\overline{yS}:y\in Y\}=\{\overline{y_0S},\ldots,\overline{y_pS}\}$
we may also suppose $h(\overline{y_0S})\leq\cdots\leq h(\overline{y_pS})$,
in this case we call $(\overline{y_0S},\cdots,\overline{y_pS})$ the {\it indicator sequence}
of $(Y,S)$ \cite{indicator}, which is a tool to classify transformation semigroups.
\\
Moreover in the transformation semigroup $(X,S)$
we say $\{\overline{xS}:x\in X\}$ with topological basis
$\{\{\overline{yS}:y\in\overline{xS}\}:x\in X\}$ is the indicator topological space of $(X,S)$.
\\
Suppose $b\in X$ and consider $X$ with topology
$\{U\subseteq X:b\notin U\vee X\setminus U$ is finite$\}$, the we
say $X$ is a Fort space with particular point $b$ \cite{counter}.
It's evident that a Fort space is just one point compactification (Alexandroff compactification)
of a discrete space.
\\
{\bf Convention.} In the following text suppose $X$ is a Fort space with particular \linebreak
point $b$.
\section{Indicator sequences and indicator topologies of Fort transformation groups with finite height}
\noindent In this section we prove that for  Fort transformation group $(X,S)$ with
$h(X,S)=n<+\infty$ there are exactly $n+1$ possible indicator sequences and $n+1$
possible non homeomorphic indicator topologies. However two Fort transformation
groups $(X,S)$, $(Y,S)$ with finite height have the same indicator sequences if and only if
they have homeomorphic indicator topologies.
\begin{lemma}\label{salam10}
In infinite Fort transformation group $(X,S)$ for all $x\in X$, the following statements are
equivalent:
\begin{itemize}
\item $x$ is a non-minimal point,
\item $xS$ is infinite,
\item $b\in\overline{xS}$ and $\{b\}$ is unique proper closed subset of $\overline{xS}$,
\item $h(\overline{xS})=1$.
\end{itemize}
\end{lemma}
\begin{proof}
Use the fact that for all $x,y\in X$ we have $xS\cap yS\neq\varnothing$ if and only if
$xS=yS$, also:
\[\overline{xS}=\left\{\begin{array}{lc} xS & xS{\rm \: is \: finite\: ,} \\
	 xS\cup\{b\} & xS{\rm \: is \: infinite\: .}\end{array}\right.\]
and $bS=\overline{bS}=\{b\}$.
\end{proof}
\begin{lemma}\label{salam20}
In infinite Fort space $X$ the collection of all
possible indicator sequences of transformation groups $(X,S)$ with
$h(X,S)=n<+\infty$ is:
\[\Xi:=\{(p_0,\cdots,p_n)\in\{0,1\}^{n+1}:0=p_0\leq p_1\leq \cdots\leq p_n=1\}\:.\]
Thus there are exactly $n$ possible indicator sequences for $(X,S)$.
\end{lemma}
\begin{proof}
Using Lemma~\ref{salam10} and the definition of indicator sequence in
infinite Fort transformation group $(X,S)$ with $h(X,S)=n<+\infty$ the indicator
sequence of $(X,S)$ belongs to $\{0,1\}^{n+1}$,
moreover if $\{\overline{xS}:x\in X\}=\{\overline{x_0S},\ldots,\overline{x_nS}\}$, then there exists $i$ such that $\overline{x_iS}$
is infinite (since $X=\bigcup\{\overline{xS}:x\in X\}=\bigcup\{\overline{x_0S},\ldots,\overline{x_nS}\}$ is infinite) which
leads to $h(\overline{x_iS})=1$. Now suppose
$(\underbrace{0,\cdots,0}_{p{\rm \: times}},\underbrace{1,\cdots,1}_{q{\rm \: times}})\in \Xi$, choose distinct $x_1=b,\cdots,x_p\in X$
and $q$ disjoint one-to-one sequences $\{x_n^i\}_{n\in\mathbb{Z}}$ in
\linebreak
$X\setminus\{x_1,\ldots,x_p\}$. Suppose $G$ is the collection of all
permutations $f:X\to X$ such that:
\[f(x)=\left\{\begin{array}{lc} x & x=x_1,\ldots,x_p \: , \\
x_{n+1}^i & x=x_n^i,2\leq i\leq q\:,\end{array}\right.\]
then:
\begin{eqnarray*}
\{\overline{xS}:x\in X\}&=&\{\overline{x_iS}:1\leq i\leq p\}\cup
\{\overline{x_0^iS}:1\leq i\leq q\} \\
&=&\{X\setminus(\{x_2,\ldots,x_p\}\cup
\{x^i_n:n\in\mathbb{Z},i=2,\ldots,q\}), \\
&& \SP\SP\SP \{b\}\cup\{x^2_n:n\in\mathbb{Z}\},\ldots,
\{b\}\cup\{x^q_n:n\in\mathbb{Z}\} , \\
&& \SP\SP\SP\SP\SP\SP\SP\SP\SP\SP\SP\SP \{b\}=\{x_1\},\ldots,\{x_p\}\}
\end{eqnarray*}
has exactly $p+q$ elements and
\begin{eqnarray*}
0 & = & h(\{x_1\})=\cdots=h(\{x_p\}) \\
& < & 1=h(\{b\}\cup\{x^2_n:n\in\mathbb{Z}\})
=\cdots=h(\{b\}\cup\{x^q_n:n\in\mathbb{Z}\}) \\
& = & h(X\setminus(\{x_2,\ldots,x_p\}\cup
\{x^i_n:n\in\mathbb{Z},i=2,\ldots,q\}))
\end{eqnarray*}
which shows $(X,S)$ has indicator sequence
$(\underbrace{0,\cdots,0}_{p{\rm \: times}},\underbrace{1,\cdots,1}_{q{\rm \: times}})$ and completes the proof.
\end{proof}
\begin{theorem}\label{salam25}
In Fort space $X$ the collection of all
possible indicator sequences of transformation groups $(X,S)$ with
$h(X,S)=n<+\infty$ is:
\[\Xi:=\left\{\begin{array}{lc} \{(p_0,\cdots,p_n)\in\{0,1\}^{n+1}:0=p_0\leq p_1\leq \cdots\leq p_n=1\} & X{\rm \: is \: infinite}\:, \\
\{(0,\ldots,0)\} & X {\rm \: is \: finite}\:.\end{array}\right.\]
In particular for infinite $X$, $h(X,S)>0$.
\end{theorem}
\begin{proof}
Use Lemma~\ref{salam20}, Lemma~\ref{salam10}, definition of indicator sequence
and the fact that in a finite transformation group all points are minimal.
\end{proof}
\begin{theorem}\label{salam50}
In Fort transformation group $(X,S)$ suppose $\alpha={\rm card}\{xS:x$ is a
minimal point$\}\setminus\{bS\}$ and $\beta={\rm card}\{xS:x$ is a
non-minimal point$\}\cup\{bS\}$, then the
indicator topological space of $(X,S)$ is homeomorphic to disjoint union of $Y_\alpha$, $Z_\beta$,
where $Y_\alpha$ is $\alpha$ under discrete topology and and $Z_\beta$ is $\beta$ under
topology $\{U\subseteq\beta:0\in U\}\cup\{\varnothing\}$.
\end{theorem}
\begin{proof}
Suppose 
\begin{center}
$\varphi:\{xS:x$ is a
minimal point$\}\setminus\{bS\}\to\alpha$
\end{center}
and 
\begin{center}
$\psi:\{xS:x$ is a
non-minimal point$\}\cup\{bS\}\to\beta$ 
\end{center}
are bijections with
$\psi(bS)=0$. Define
$\eta:\{\overline{xS}:x\in X\}\to Y_\alpha\sqcup Z_\beta$ with
\[\eta(\overline{xS})=\left\{\begin{array}{lc}\varphi(xS) & x
{\rm \: is \: a \: minimal \: point \: and \:}x\notin bS\:, \\
\psi(xS) & x
{\rm \: is \: a \: non-minimal \: point \: or \:} x\in bS \:. \end{array}\right.\]
Using Lemma~\ref{salam10} it is easy to see that
$\eta:\{\overline{xS}:x\in X\}\to Y_\alpha\sqcup Z_\beta$ is a
homeomorphism, where $W:=\{\overline{xS}:x\in X\}$
considered with induced indicator
topology $\{V\subseteq W:\forall x,y\in X\:(
\overline{xS}\in W\wedge\overline{yS}\subseteq\overline{xS}\Rightarrow
\overline{yS}\in W)\}$, which leads to the desired result.
\end{proof}
\begin{corollary}\label{salam40}
Using the same notations as in Theorem~\ref{salam50},
in Fort transformation group $(X,S)$ with indicator sequence
$(\underbrace{0,\cdots,0}_{p{\rm \: times}},\underbrace{1,\cdots,1}_{q{\rm \: times}})$
the indicator topological space is homeomorphic to disjoint union of $Y_{p-1}$, $Z_{q}$. In particular, every two Fort transformation groups
with finite height and same indicator sequence, have homeomorphic
indicator topological spaces.
\end{corollary}
\begin{corollary}\label{salam30}
For homeomorphism $f:X\to X$ let $T_f:=\{f^n:n\in{\mathbb Z}\}$,
then:
\begin{itemize}
\item[1.] the collection $A:=\{h(X,T_g):g:X\to X$ is a homeomorphism
	$\}$ is equal to:
	\[\left\{\begin{array}{lc} \{0,1,\ldots,|X|-1\} & X {\rm \: is \: finite \: ,} \\
	\{1,2,\ldots\}\cup\{\infty\} & X {\rm \: is \: infinite \: countable \: ,} \\
	\{\infty\} & X {\rm is \: uncountable \:,} \end{array}\right.\]
\item[2.] for countable $X$ the collection of all possible indicator
	sequences of $(X,T_g)$ for some homeomorphism $g:X\to X$,
	is the same as the collection of all possible indicator
	sequences of $(X,S)$ for some group $S$,
\item[3.] for countable $X$ the collection of all possible indicator
	topological spaces of $(X,T_g)$ for some homeomorphism $g:X\to X$,
	is the same as the collection of all possible indicator
	topological spaces of $(X,S)$ for some group $S$.
\end{itemize}
\end{corollary}
\begin{proof}
{\bf 1.} First suppose there exists homeomorphism $g:X\to X$ with $h(X,T_g)=n<+\infty$, so there are $x_0,\ldots,x_n\in X$ with
$\{\overline{xT_g}:x\in X\}=\{\overline{x_iT_g}:i=0,\ldots,n\}$.
Therefore (use Lemma~\ref{salam10}):
\begin{eqnarray*}
X & = & \bigcup\{\overline{xT_g}:x\in X\} \\
& = & \bigcup\{\overline{x_iT_g}:i=0,\ldots,n\} \\
& = & \{f^j(x_i):j\in{\mathbb Z},i=0,\ldots,n\}\cup\{b\}
\end{eqnarray*}
and $X$ is countable, so
for uncountable $X$, $A=\{\infty\}$.
\\
Now suppose $X=\{x_0,\ldots,x_n\}$ is finite with $n+1$ elements, then for all homeomorphism $f:X\to X$
we have $h(X,T_f)+1=|\{\overline{xT_f}:x\in X\}|\leq |X|$. Moreover for
$i\leq n$, define permutation $g:X\to X$ with
	\[g(x)=\left\{\begin{array}{lc} x & x\in X\setminus\{x_j:j\geq i\} \:, \\
	x_{j+1} & x=x_j\wedge i\leq j<n\:, \\
	x_{i} & x=x_n\:. \end{array} \right. \]
Then $\{\overline{xT_g}:x\in X\}=\{xT_g:x\in X\}=
\{\{x_i,x_{i+1},\ldots,x_n\}\}\cup\{\{x\}:x\in X\setminus\{x_j:j\geq i\}\}$
and has exactly $i+1$ elements, so $h(X,T_g)=i$, and in this case
$A=\{0,\ldots,n\}$.
\\
Finally suppose $X$ is infinite countable.
In this case $\{b\}\subset X$ are two closed invariant subset of $(X,T_f)$ and $h(X,T_f)\geq1$ (for all homeomorphism $f:X\to X$). Morever
$h(X,T_{id_X})=\infty$. Consider $i\in\{1,2,\ldots\}$ and choose $i$
distinct points $b=x_0,x_1,\ldots,x_{i-1}\in X$ and one-to-one sequence
$\{y_n\}_{n\in\mathbb{Z}}$ with 
\[\{y_n:n\in\mathbb{Z}\}=X\setminus
\{x_0,x_1,\ldots,x_{i-1}\}\:,\]
then for homeomorphism $g:X\to X$
with
\[g(x)=\left\{\begin{array}{lc} x & x=x_0,x_1,\ldots,x_{i-1}\:, \\
y_{n+1} & x=y_n,n\in\mathbb{Z}\:, \end{array}\right.\]
we have $\{\overline{xT_g}:x\in X\}=\{\{x_0\},\{x_1\},\ldots,\{x_{i-1}\},
\{y_n:n\in\mathbb{Z}\}\cup\{b\}\}$ has exactly $i+1$ elements
and $h(X,T_g)=i$, which completes the proof of item (1).
\\
{\bf 2.} For finite $X$ and Group $S$ with $h(X,S)=i$ the only
possible indicator sequence is $(\underbrace{0,\cdots,0}_{i+1{\rm \: times}})$ with $i\leq |X|-1$, using the same method described in the
proof of item (1), for all $i\leq |X|-1$ there exists bijection (homeomorphism) $g:X\to X$ with $h(X,T_g)=i$ and indicator sequence $(\underbrace{0,\cdots,0}_{i+1{\rm \: times}})$.
\\
Now suppose $X$ is infinite countable and for $p,q\geq1$
choose distinct points $b=x_1,x_2,\ldots,x_p\in X$ and disjoint
one-to-one sequences $\{y^1_n\}_{n\in\mathbb{Z}},\ldots,
\{y^q_n\}_{n\in\mathbb{Z}}$ with $\{y^j_n:1\leq j\leq q,n\in\mathbb{Z}\}=
X\setminus\{x_1,\ldots,x_p\}$. Define homeomorphism
$g:X\to X$ with:
\[g(x)=\left\{\begin{array}{lc} x & x=x_1,x_2,\ldots,x_p\:, \\
y^j_{n+1} & x=y^j_n,n\in\mathbb{Z},1\leq j\leq q\:, \end{array}\right.\]
using a similar method described in the proof of Lemma~\ref{salam20}
we have $h(X,T_g)=p+q-1$ and $(\underbrace{0,\cdots,0}_{p{\rm \: times}},\underbrace{1,\cdots,1}_{q{\rm \: times}})$ is the indicator sequence of $(X,T_g)$. Using Theorem~\ref{salam25} we have the desired result.
\\
{\bf 3.} Use item (2) and Corollary~\ref{salam40}.
\end{proof}
\section{Indicator sequences and indicator topologies of Fort transformation semigroups with finite height}
\noindent In this section we characterize and study indicator sequences and indicator topologies of
Fort transformation semigroups. Finally we complete the text with counterexamples
comparing our results on Fort transformation groups and Fort transformation semigroups.
\begin{theorem}\label{salam60}
For finite topological space $W$, the following statements are equivalent:
\begin{itemize}
\item[1.] there exists transformation semigroup $(K,H)$ with finite discrete $K$, whose indicator topological space
and $W$ are homeomorph,
\item[2.] there exists Fort transformation semigroup $(X,S)$  whose indicator topological space
and $W$ are homeomorph,
\item[3.] there exists transformation semigroup $(Z,T)$  whose indicator topological space
and $W$ are homeomorph.
\end{itemize}
\end{theorem}
\begin{proof}
``(3) $\Rightarrow$ (1)'' Suppose
there exists transformation semigroup $(Z,T)$  whose indicator topological space
and $W$ are homeomorph. Let $K:=\{\overline{xT}:x\in Z\}$
with discrete topology (note that $|K|=|W|<\infty$), and
$H:=\{f\in K^K:\forall w\in K\:(wf\subseteq w)\}$. Then
it is clear that $K$ is a semigroup of self-maps on $K$ (under the
operation of composition) containing identity map of $K$, moreover:
\[\forall k\in K\SP  kH=\{w\in K:w\subseteq k\}\:, \tag{*}\]
consider $\psi:\mathop{\{\overline{xT}:x\in Z\}\to\{kH:k\in K\}}\limits_{k
\mapsto kH}$, it's evident that this map is onto. For $x,y\in Z$ we have:
\begin{eqnarray*}
\psi(\overline{xT})=\psi(\overline{yT}) & \Rightarrow &
	\overline{xT}\:H=\overline{yT}\: H \\
& \Rightarrow & \overline{xT}\in\overline{yT}\: H\wedge
	\overline{yT}\in\overline{xT}\: H \\
& \mathop{\Rightarrow}\limits^{{\rm by \:}(*)} &  \overline{xT}\subseteq\overline{yT}
	\wedge  \overline{yT}\subseteq\overline{yT} \\
& \Rightarrow &  \overline{xT}=\overline{yT}
\end{eqnarray*}
and $\psi$ is one-to-one. Moreover, for
$A\subseteq \{\overline{xT}:x\in Z\}$ we  have:
\\
$A$ is an open subset of indicator topological space of $(Z,T)$
\begin{eqnarray*}
& \Leftrightarrow & \forall\overline{xT}\in A\:\forall\overline{yT}\subseteq
\overline{xT}\:\overline{yT}\in A \\
& \Leftrightarrow & \forall\psi(\overline{xT})\in \psi(A)\:\forall \psi(\overline{yT})\subseteq
\psi(\overline{xT})\:\psi(\overline{yT})\in \psi(A) \\
& \Leftrightarrow & \forall kH\in \psi(A)\:\forall lH\subseteq
kH\:lH\in \psi(A) \\
&\Leftrightarrow & \psi(A) {\rm \:
is \: an \: open \: subset \: of \: indicator \: topological \: space \: of \:}
(K,H)
\end{eqnarray*}
Thus indicator topological spaces of $(Z,T)$ and $(K,H)$ are
homeomorph, which completes the proof.
\end{proof}
\begin{corollary}\label{salam65}
For finite sequence $(p_0,\ldots,p_n)$ of integers the following
statements are equivalent:
\begin{itemize}
\item[1.] there exists finite discrete transformation semigroup $(K,H)$ whose indicator sequence is $(p_0,\ldots,p_n)$,
\item[2.] there exists Fort transformation semigroup $(X,S)$  whose indicator sequence is $(p_0,\ldots,p_n)$,
and $W$ are homeomorph,
\item[3.] there exists transformation semigroup $(Z,T)$  whose sequence is $(p_0,\ldots,p_n)$.
\end{itemize}
\end{corollary}
\begin{proof}
``(3) $\Rightarrow$ (1)'' Suppose
there exists transformation semigroup $(Z,T)$  whose indicator sequence is $(p_0,\ldots,p_n)$, using Theorem~\ref{salam60}
there exists finite discrete transformation semigroup $(K,H)$ whose indicator
topological space and indicator topological space of $(Z,T)$ are homeomorph
(note that the indicator topological space of $(Z,T)$ has $h(Z,T)+1(<\infty)$
elements). Since they have homeomorphic indicator topological spaces,
they have the same indicator sequence (use~\cite[Note 5.13]{indicator}).
\end{proof}
\begin{theorem}\label{salam70}
For self-map $f:A\to A$ let $S_f:=\{f^n:n\geq0\}$.
For finite sequence $(p_0,\ldots,p_n)$ of integers the following
statements are equivalent:
\begin{itemize}
\item[1.] there exists finite discrete space $K$ and $f:K\to K$ such that $(p_0,\ldots,p_n)$ is the indicator sequence of $(K,S_f)$,
\item[2.] there exists Fort space $X$ and continuous map $f:X\to X$ such that $(p_0,\ldots,p_n)$ is the indicator sequence of $(X,S_f)$,
\item[3.] $p_0=0$ and for all $i\in\{1,\ldots,n\}$ we have $0\leq p_i-p_{i-1}\leq1$.
\end{itemize}
Moreover the collection of all sequences like $(p_0,\ldots,p_n)$
such that $p_0=0$ and $0\leq p_i-p_{i-1}\leq1$ for all $i\in\{1,\ldots,n\}$,
has $2^n$ elements.
\end{theorem}
\begin{proof}
``(2) $\Rightarrow$ (3)''  Suppose $(p_0,\ldots,p_n)$ is the indicator sequence of $(X,S_f)$ Fort space $X$ and continuous map $f:X\to X$,
then there exist $x_0,\ldots,x_n\in X$ with
$p_0=h(\overline{x_0S_f})\leq\cdots\leq p_n=h(\overline{x_nS_f})$
and $\{\overline{xS_f}:x\in X\}=\{\overline{x_iS_f}:0\leq i\leq n\}$ has
exactly $n+1$ elements. Moreover since $(X,S_f)$ has minimal elements,
we have $p_0=0$. Now suppose $p_i=h(\overline{x_iS_f})>0$,
then we have the following cases:
\begin{itemize}
\item $x_iS_f=\{x_if^n:n\geq0\}=\{x_i,x_if,\ldots,x_if^k\}$ has exactly $k+1$
elements, in this case $x_iS_f=\overline{x_iS_f}$ and since
$h(\overline{x_iS_f})>0$, $\overline{x_iS_f}$ is not minimal.
Thus $x_i$ is not a periodic point of $f$. So $\{x_if,\ldots,x_if^k\}$ not only
is a closed invariant proper subset of $\{x_i,x_if,\ldots,x_if^k\}$ but also it contains all of closed invariant proper subsets of $\{x_i,x_if,\ldots,x_if^k\}$.
Thus 
$h(\overline{x_ifS_f})=h(\{x_if,\ldots,x_if^k\})=
h(\{x_i,x_if,\ldots,x_if^k\})-1=p_i-1$. Therefore $p_i-1\in\{h(\overline{xS_f}):x\in X\}=\{p_0,\ldots,p_n\}$, using $p_0\leq p_1\leq\cdots\leq p_n$,
there exists $j<i$ with $p_j=p_i-1\leq p_{j+1}\leq\cdots\leq p_i$
which leads to $0\leq p_i-p_{i-1}\leq1$.
\item $x_iS_f=\{x_if^n:n\geq0\}$ is infinite. In this case $\{x_if^n\}_{n\geq0}$
is a one-to-one sequence and $\{\overline{x_if^nS_f}\}_{n\geq0}$
($=\{\{x_if^k:k\geq n\}\cup\{b\}\}_{n\geq0}$) is a one-to-one
decreasing sequence of closed invariant subsets of
$\overline{x_iS_f}$, which leads to the contradiction
$p_i=h(\overline{x_iS_f})=\infty$.
\end{itemize}
Using the above two cases we have $0\leq p_i-p_{i-1}\leq1$.
\\
``(3) $\Rightarrow$ (1)'' Consider $0=p_0\leq p_1\leq\cdots\leq p_n$ with
$0\leq p_i-p_{i-1}\leq1$ for all $i\in\{1,\ldots,n\}$, and
discrete space $K=\{x_0,\ldots,x_n\}$ with $n+1$ elements.
Define $g:K\to K$ with:
\[g(x)=\left\{\begin{array}{lc} x & x=x_i, p_i=0, i\in\{0,\ldots,n\}\:, \\
x_t & x=x_i, p_i>0, t=\max\{j:p_j<p_i\},  i\in\{0,\ldots,n\}\:,
\end{array}\right.\]
we prove the indicator sequence of $(K,S_g)$ is $(p_0,\ldots,p_n)$,
for this aim we use the following two claims:
\\
{\it Claim 1}. For all $i\in\{0,\ldots,n\}$, $h(\overline{x_iS_g})=h(x_iS_g)=p_i$. Consider $i\in\{0,\ldots,n\}$, if $p_i=0$, then $\overline{x_iS_g}=x_iS_g=\{x_i\}$ and it's evident
that $h(\overline{x_iS_g})=0=p_i$. Suppose $p_i>0$ and for all $j<i$ we have $h(\overline{x_jS_g})=p_j$, moreover $t=\max\{j:p_j<p_i\}$,
then $\overline{x_iS_g}=x_iS_g=\{x_i\}\cup x_tS_g$ also it's clear that
for all $k,j\in\{0,\ldots,n\}$ with $g(x_j)=x_k$ we have $k\leq j$, hence
$x_i\notin \{x_0,\ldots,x_t\}\supseteq x_tS_g$. Now using $h(x_tS_g)=p_t$
the collection $\{xS_g:x\in x_tS_g\}$ has $p_t+1$ elements. Thus
$\{xS_g:x\in x_iS_g\}=\{xS_g:x\in x_tS_g\}\cup\{x_iS_g\}$ has
$p_t+2=p_i+1$ elements and $h(x_iS_g)=p_i$.
\\
{\it Claim 2}. For distinct $s,t\in\{0,\ldots,n\}$, we have
$x_sS_g\neq x_tS_g$. Consider  $s,t\in\{0,\ldots,n\}$ with $s<t$.
Since for all $k,j\in\{0,\ldots,n\}$ with $g(x_j)=x_k$ we have $k\leq j$,
we have $x_tSg\setminus x_sS_g\supseteq\{x_t\}\setminus\{x_0,\ldots,x_s\}=\{x_t\}$, which leads to the desired result.
\\
Using Claim 2, $\{\overline{xS_g}:x\in K\}=\{xS_g:x\in K\}=
\{x_iS_g:0\leq i\leq n\}$ has $n+1$ elements, which completes the
proof by Claim 1.
\end{proof}
\begin{counterexample}
Using Corollary~\ref{salam40} and Theorem~\ref{salam60}, for topological space
$W:=\{1,2,3\}$ with topology $\tau:=\{\{1\},\{1,2\},\{1,2,3\},\varnothing\}$:
\begin{itemize}
\item there is not any Fort transformation group whose indicator topological space is homeomorph with $(W,\tau)$,
\item there exists a Fort transformation semigroup whose indicator topological space is homeomorph with $(W,\tau)$.
\end{itemize}
Using Theorem~\ref{salam25} and Corollary~\ref{salam65}:
\begin{itemize}
\item there is not any Fort transformation group whose indicator sequence is $(0,1,2)$,
\item there exists a Fort transformation semigroup whose indicator sequence is $(0,1,2)$
(and its indicator topological space is homeomorph with $(W,\tau)$).
\end{itemize}
\end{counterexample}


\noindent {\small {\bf Fatemah Ayatollah Zadeh Shirazi},
Faculty of Mathematics, Statistics and Computer Science,
College of Science, University of Tehran ,
Enghelab Ave., Tehran, Iran
\\
({\it e-mail}: fatemah@khayam.ut.ac.ir)
\\
{\bf Fatemeh Ebrahimifar},
Faculty of Mathematics, Statistics and Computer Science,
College of Science, University of Tehran ,
Enghelab Ave., Tehran, Iran
\\
({\it e-mail}: ebrahimifar64@ut.ac.ir)}


\begin{thebibliography}{99}

\bibitem{indicator} F. Ayatollah Zadeh Shirazi,
N. Golestani, On classifications of transformation semigroups: indicator sequences and indicator topological spaces, Filomat, 2012 (26:2), 313--329

\bibitem{nemidoonam1}   T. Karube, Transformation groups satisfying some local metric conditions, J. Math. Soc. Japan, 1966 (18), 45--50

\bibitem{nemidoonam2} N. T. Nhu, The group of measure preserving transformations of the unit interval is an absolute retract,
Proc. Amer. Math. Soc., 1990 (110, no. 2), 515--522.

\bibitem{nemidoonam3}  V. G. Pestov,  A topological transformation group without non-trivial equivariant compactifications, Adv. Math., 2017 (311), 1--17

\bibitem{height} M. Sabbaghan, F. Ayatollah Zadeh Shirazi,
Transformation semigroups and transformed dimensions, Journal of Sciences, Islamic Republic of Iran, 2001 (12, no.1), 65--75

\bibitem{counter} L. A. Steen, J. A. Seebach, Counterexamples in topology, New York-Montreal, Holt, Rinehart and Winston Inc., 1970

\end{thebibliography}
\end{document}